\documentclass{amsproc}

\usepackage{xcolor}
\usepackage{latexsym, epsfig, epic, xcolor, upgreek}
\usepackage{euler, amsfonts, amssymb, amsmath, amsthm, amsopn}
\usepackage{accents}
\usepackage{hyperref}
\newtheorem{theorem}{Theorem}[section]

\newtheorem{corollary}[theorem]{Corollary}
\newtheorem{question}[theorem]{Question}

\theoremstyle{definition}
\newtheorem{definition}[theorem]{Definition}
\newtheorem{example}[theorem]{Example}

\theoremstyle{remark}
\newtheorem{remark}[theorem]{Remark}
\newtheorem{property}[theorem]{Property}

\numberwithin{equation}{section}

\begin{document}

\title{On Positivity for the Peterson Variety}

%    Information for first author
\author{Rebecca Goldin}
%    Address of record for the research reported here
\address{Department of Mathematical Sciences MS 3F2, George Mason University, Fairfax, VA 22030}
%    Current address
\curraddr{Department of Mathematical Sciences MS 3F2, George Mason University, Fairfax, VA 22030}
\email{rgoldin@gmu.edu}
%    \thanks will become a 1st page footnote.
\thanks{The author was supported in part by NSF Grant \#2152312}

%    Information for second author

%    General info
\subjclass{Primary}
\date{August 31, 2022}

\dedicatory{}


\begin{abstract}
We aim in this manuscript to describe a specific notion of {\em geometric positivity} that manifests in cohomology rings associated to the flag variety $G/B$ and, in some cases, to subvarieties of $G/B$. We offer an exposition on the the well-known geometric basis of the homology of $G/B$ provided by Schubert varieties, whose dual basis in cohomology has nonnegative structure constants. 
In recent work \cite{GolMihSin21} we showed that the equivariant cohomology of Peterson varieties satisfies a positivity phenomenon similar to that for Schubert calculus for $G/B$. Here we explain how this positivity extends to this particular nilpotent Hessenberg variety, and offer some open questions about the ingredients for extending positivity results to other  Hessenberg varieties.
\end{abstract}

\maketitle

\newtheorem*{poetry}{}

\renewcommand\eqref[1]{(\ref{#1})}
\newcommand\MMN{{\mathcal M}_N}
\newcommand\A{\textsc{a}}
\newcommand\B{\textsc{b}}
\newcommand\notdivides{\,\not\!\divides\,}
\newcommand\divides{\,{\mid}\,}
\newcommand\lie[1]{{\mathfrak{#1}}}
\newcommand\lien{{R_N^{\Delta=0}}}
\newcommand\Pf{{\em Proof. }}
\newcommand\Tr{{\rm Tr\,}}
\newcommand\Proj{{\rm Proj\,}}
\newcommand\iso{{\ \cong\ }}
\newcommand\tensor{{\otimes}}
\newcommand\Tensor{{\bigotimes}}
\newcommand\QED{\hfill $\Box$} %{\bf QED}}
\newcommand\calO{{\mathcal O}}
\newcommand\calC{{\mathcal C}}
\newcommand\calS{{\mathcal S}}
\newcommand\calT{{\mathcal T}}
\newcommand\Pet{{\bf P}}
\newcommand\frob{\bowtie}
\newcommand\onto{\mathop{\twoheadrightarrow}}
\newcommand\into{\operatorname*{\hookrightarrow}}
\newcommand\Ex{\noindent{\em Example. }}
\newcommand\union{\bigcup}
\newcommand\Pone{{\mathbb P}^1}
\newcommand\pt{pt}
\newcommand\calL{{\mathcal L}}
\newcommand\Id{{\bf 1}}
\newcommand\Td{{\rm Td}}
\newcommand\vol{{\rm vol}}
\newcommand\CP{{\mathbb C \mathbb P}}
\newcommand\PP{{\mathbb P}}
\newcommand\ind{{\rm index}}
\newcommand\reals{{\mathbb R}}
\newcommand\complexes{{\mathbb C}}
\newcommand\CC{{\mathbb C}}
\newcommand\integers{{\mathbb Z}}
\newcommand\rationals{{\mathbb Q}}
\newcommand\onehalf{\frac{1}{2}}
\newcommand\naturals{{\mathbb N}}
\newcommand\Star{\operatornamewithlimits\bigstar}
\newcommand\dec{\hbox{\ \rm\tiny decreasing}}
\newcommand\area{\operatorname{area}}
\newcommand\ol{\overline}
\newcommand\Grmn{{\rm Gr}^m(\complexes^n)}
\newcommand\<{\langle}
\renewcommand\>{\rangle}
\newcommand\junk[1]{}
\newcommand\To{\longrightarrow}
\newcommand\Coprod{\coprod}
\newcommand\Union{\bigcup}
\newcommand\dash{\text{--}}

\newcommand{\transv}{\mathrel{\text{\tpitchfork}}} % For transverse intersections
\makeatletter
\newcommand{\tpitchfork}{%
  \vbox{
    \baselineskip\z@skip
    \lineskip-.48ex
    \lineskiplimit\maxdimen
    \m@th
    \ialign{##\crcr\hidewidth\smash{$-$}\hidewidth\crcr$\pitchfork$\crcr}
  }%
}
\makeatother

\newcommand\dfn{\bf} % maybe should be \em
\newcommand\indentme{\parindent 2 cm}
\newcommand\GLN{{{GL_N}}}
\newcommand\glN{\lie{gl}_N}
\newcommand\SLN{{{SL_N}}}%(\complexes)}}}
\newcommand\FlN{{Flags(\vec N)}}

\newcommand\thmlabel[1]{\label{thm:#1}}
\newcommand\proplabel[1]{\label{prop:#1}}
\newcommand\seclabel[1]{\label{sec:#1}}
\newcommand\sseclabel[1]{\label{ssec:#1}}
\newcommand\lemlabel[1]{\label{lem:#1}}
\newcommand\figlabel[1]{\label{fig:#1}}

\newcommand\thmref[1]{theorem \ref{thm:#1}}
\newcommand\propref[1]{proposition \ref{prop:#1}}
\newcommand\secref[1]{\S\ref{sec:#1}}
\newcommand\ssecref[1]{\S\ref{ssec:#1}}
\newcommand\lemref[1]{lemma \ref{lem:#1}}
\newcommand\figref[1]{figure \ref{fig:#1}}

\newcommand\Thmref[1]{Theorem \ref{thm:#1}}
\newcommand\Propref[1]{Proposition \ref{prop:#1}}
\newcommand\Secref[1]{\S\ref{sec:#1}}
\newcommand\Ssecref[1]{\S\ref{ssec:#1}}
\newcommand\Lemref[1]{Lemma \ref{lem:#1}}
\newcommand\Figref[1]{Figure \ref{fig:#1}}

%tentative defs
\newcommand\cycinside[1]{#1}
\newcommand\cyc[1]{\circlearrowleft (\cycinside{#1})}
\newcommand\cp{\bullet} % new product
\newcommand\der{\partial}
\newcommand\mdeg{{\rm m}\!\deg}
\newcommand\codim{{\rm co}\!\dim}
\newcommand\SN{\mathcal{S}_N}
\newcommand\MNC{{\text{Mat}_N}}%(\complexes)}}
\newcommand\MnC{{M_n}}%(\complexes)}}
\newcommand\rank{{\rm rank\ }}
\newcommand{\comment}[1]{$\star${\sf\textbf{#1}}$\star$}
\newcommand\Br{\mathcal{B}r}
\newcommand\wt[1]{{\widetilde #1}}
\newcommand{\todo}[2][0.9]{\vspace{1 mm}\par \noindent
 \framebox{\color{red}\begin{minipage}[c]{#1
\textwidth} \tt #2 \end{minipage}}\vspace{1 mm}\par}

\newcommand\AKrem[1]{{\color{teal} #1}}
\newcommand\acirc{{\mathring a}}
\newcommand\dcirc{{\mathring d}}
\newcommand\deltacirc{{\accentset{\circ} \delta}}
\newcommand\deltop{{\overset{\tiny\Leftcircle} \delta}}

\newcommand\bardelta{{\overline\delta}}
\newcommand\Demopisobaric{\delta}
\newcommand\Demop{\deltacirc}

\newcommand\ea{e^\alpha}
\newcommand\ema{e^{-\alpha}}
\newcommand\da{\Demopisobaric_\alpha}
\newcommand\bda{\Demop_\alpha}
\newcommand\ra{r_\alpha}

%{ \small\tableofcontents}

\section{Introduction to Geometric Positivity for $G/B$ and Other Varieties}
The term {\em positivity}, which may sometimes be more appropriately termed ``nonnegativity" can have many different manifestations and meanings in the context of combinatorics and algebraic geometry. We aim in this manuscript to describe a specific notion of {\em geometric positivity} that manifests in cohomology rings associated to the flag variety $G/B$ and, in some cases, to subvarieties of $G/B$. In Section~\ref{se:positivityG/B} we review the well-known geometric basis of homology provided by Schubert varieties, whose dual basis in cohomology has nonnegative structure constants. In fact, these structure constants arise as the number of points occurring in a certain 0-dimensional transverse intersection of complex algebraic varieties. While the structure constants are nonnegative, a combinatorially positive formula for them is known only in some cases. The classical Littlewood-Richardson rule expresses the Littlewood-Richardson coefficients as the number of Young tableaux that satisfy certain properties; equivariant structure constants for the Grassmannian of $k$-planes in $\mathbb C^n$ can be counted positively using Knutson-Tao puzzles. Schubert calculus is concerned both with finding combinatorially positive formulas for these structure constants (and their generalizations) as well as with proving that structure constants are  nonnegative (or have an analogous positivity property). 

Positivity for the integral cohomology ring of $G/B$ has a generalization in the $T$-equivariant cohomology of $G/B$,  given in Definition~\ref{def:positivity}. 
We assume that we have a distinguished set of positive roots $\alpha_1, \dots, \alpha_d$, or monomials in the equivariant cohomology of a point, where $d=\dim T$. For a $T$-invariant subvariety $X$ of $G/B$, a basis $\{x_i\}$  of $H_T^*(X)$ as a module over $H_T^*(pt)$ is considered {\em Graham positive} if 
$$
x_i x_j = \sum_k a_{ij}^k x_k
$$
defines structure constants $a_{ij}^{k}\in H_T^*(pt)$ that are nonnegative linear combinations of monomials in $\alpha_1, \dots, \alpha_d$. A wonderful consequence of Graham positivity for equivariantly formal spaces (a technical condition satisfied by $G/B$ and Peterson varieties) is that it implies positivity in the usual sense, as the constant part of the polynomial in $a_{ij}^k$ are the coefficients that arise in the product of ordinary cohomology classes (see Property~\ref{prop:formality}).

The basis $\{x_i\}$ is the dual basis to a set of $T$-invariant varieties forming a basis for the equivariant homology of $X$. When a basis $\{x_i\}$ is both Graham positive and also the dual of homology classes associated to $T$-invariant subvarieties, we say $X$ satisfies  {\em geometric positivity}. This phenomenon occurs for $X=G/B$, as the Schubert classes $\{\sigma_w: w\in W\}$ multiply with nonnegative structure constants, and arise as the dual basis for $\{[X_w]: w\in W\}$. 

The goals of this manuscript are to provide an exposition on the well-known {\em geometric positivity for equivariant cohomology} for $G/B$ that offers the ingredients to apply it to certain subvarieties of $G/B$. 
We focus specifically on a subvariety $\Pet$ called the {\em Peterson variety} (Definition~\ref{def:Pet}). The Peterson variety was introduced by Dale Peterson, but appeared in published work by Kostant \cite{Kos96}, and Rietch \cite{Rie01} \cite{Rie03} on the quantum cohomology of flag varieties. 

Peterson varieties arise as a special case of a large class of varieties called Hessenberg varieties (Definition \ref{def:Hess}), which are themselves parameterized by an element $x$ in the Lie algebra $\mathfrak g$ of $G$, and Borel-invariant subspace $H$ of $\mathfrak g$. %They are the Hessenberg variety when $H$ is the smallest subspace of $\mathfrak g$ containing $\mathfrak b$ and the negative simple weight spaces, and $x$ is a principle nilpotent element of $\mathfrak g$. 

Hessenberg varieties were first defined in generality by De Mari, Procesi and Shayman \cite{deMProSha92}, who explored the geometry of regular semisimple Hessenberg varieties.  The cohomology rings of regular nilpotent Hessenberg varieties were studied in, for example \cite{AbeHarHorMas19}, while the Poincar\'e polynomials and geometric bases were studied in \cite{Pre17} and \cite{EnoHorNagTsu22}. In \cite{Kly85}, \cite{Kly95} and \cite{AbeFujZen20}, the authors proved and exploited the structure of flat families of Hessenberg varieties.
Other important special cases occur when $H$ is restricted: Springer varieties, whose geometry was explored in \cite{GorKotMac06}, arise as Hessenberg varieties when $H=\mathfrak b$, the Lie algebra of the Borel $B$. 
 Tymoczko \cite{Tym07} proved that all Hessenberg varieties are paved by affines in Type A, while Precup \cite{Pre13} extended these results for regular elements to other Lie types.  Hessenberg varieties are also deeply connected to combinatorics; 
a recent survey on many combinatorial aspects is \cite{AbeHor20} and the references therein.  

The Peterson variety admits an action of a one-dimensional torus $S \subset T$ with finitely many fixed points. The subgroup $S$ occurs as those elements of $T$ that stabilize a principal nilpotent element of $\mathfrak g$; in particular, $\alpha_i|_S$ is independent of $i$ for all positive simple roots $\alpha_i$. 
Harada, Horiguchi and Masuda proved a presentation of $H_S^*(\Pet;\mathbb Q)$ in \cite{HarHorMas15} unrelated to Schubert calculus. In \cite{HarTym11} and \cite{Dre15}, the authors proved that there is a basis of the equivariant cohomology given by pulling back to $\Pet$ a specific set of Schubert classes from $H_S^*(G/B)$. They proved a  positive equivariant Monk formula for this basis, but did not provide a geometric basis.

As we shall see, an
essential ingredient to positivity for the equivariant cohomology of $G/B$ is that it has a paving by $B$-invariant (or $B$-stable) affine cells.
The property of being paved by affines is enjoyed by regular Hessenberg varieties, including the Peterson variety; however the $B$-invariance is not. Despite this missing ingredient, we show here (originally proved by the author together with Mihalcea and Singh, in \cite{GolMihSin21}) that Peterson varieties have equivariant geometric positivity using properties of $G/B$ and the inclusion of the Peterson variety into $G/B$. More specifically, we show that there exists a basis $\{p_K\}_{K\subset \Delta}$ of $H_S^*(\Pet)$, where the indexing set is the set of subsets of the simple positive roots, making $H_S^*(\Pet)$ Graham positive. In particular,
$$
p_Ip_J = \sum_K b_{IJ}^K p_K
$$ for monomials $b_{IJ}^K \in H_S^*(pt)$ that are nonnegative multiples of a power of $t := \alpha_i|_S$. Furthermore, this basis occurs as a dual basis to a set of $S$-invariant subvarieties $P_K$ whose homology classes are a basis for the equivariant homology of $\Pet$. In type A, the author and Gorbutt \cite{GolGor22} proved positivity combinatorially, i.e. they found a manifestly positive formula for the structure constants $b_{IJ}^K$. 

In this manuscript we focus on the positivity emerging from the product structure of a specially chosen basis of cohomology. 
We establish notation and basic definitions in Section~\ref{se:defs}, followed by a brief explanation of key properties satisfied by the equivariant cohomology of $G/B$ in Section~\ref{se:ecohomologyG/B}, including an alternate notion of positivity obtained through localization. We recall some geometric properties of Hessenberg varieties in Section~\ref{se:paving}, and explain the reasons behind geometric positivity for $G/B$ in Section~\ref{se:positivityG/B}, for both ordinary and equivariant cohomology. Section~\ref{se:positivityPeterson} explores the geometry of $\Pet$ and demonstrates that Peterson Schubert calculus satisfies geometric positivity.  Finally, in Section~\ref{se:Hess} we  share some open questions about whether the analogous cohomology basis for  other nilpotent Hessenberg varieties have the same positivity property.

The author thanks an early referee for useful suggestions. This work was partially supported by NSF Grant \#2152312.
%The main results in this manuscript are largely based on joint work with Mihalcea and Singh \cite{}. 

\section{Notation and Definitions}\label{se:defs}
For the purposes of this manuscript, we let $G$ be a complex semi-simple Lie group, $B$ a Borel subgroup, and $B^-$ an opposite Borel.  Denote by $\mathfrak g$ the Lie algebra of $G$, and $\mathfrak b$ the Lie algebra of $B$. Let $T = B\cap B^-$ be the maximal torus obtained by the intersection, and $\mathfrak h = Lie(T)$ its Lie algebra. We let $W = N(T)/T$ denote the Weyl group for $G$,  $\Delta$ the set of simple positive roots, and $\phi^+$ the set of positive roots.  There is a length function $\ell: W\rightarrow \mathbb N$ sending each Weyl group element to the minimal number of  reflections over simple roots required to describe that element; there is a unique longest-length element of $W$, which we denote by $w_0$.

We call $G/B$ the {\em flag manifold}. Observe that $T$ acts on $G/B$ the left by group multiplication, which is well-defined on the level of cosets; the fixed points of this action are isolated, and we denote the set of them by $(G/B)^T$. We will frequently identify them with the Weyl group, via the identification 
$$\tilde{v}B\leftrightarrow v\in W,$$
where $\tilde{v}$ is any lift of $v$ in $N(T)$.  Since the coset doesn't depend on the lift, we abuse notation and call the coset $vB$. 

For each $v\in W$, the closure of the $B$ orbit on $vB$ is given by $X_v = \overline{BvB/B}$. These are called the {\em Schubert varieties} indexed by $v\in W$. Similarly, $X^v =  \overline{B^-vB/B}$ are the closures of the $B^-$ orbit on $vB$, and are called {\em opposite Schubert varieties}.

Our initial focus is on a specific regular, nilpotent case, given as follows. Let $e \in \bigoplus_{\alpha\in \Delta} \mathfrak g_\alpha$ be a principal nilpotent element, where $\mathfrak g_\alpha$ is the weight space of the simple positive root $\alpha$ in $\mathfrak g$. Let $G^e$ be the centralizer of $e$ in $G$.
\begin{definition}\label{def:Pet}
The Peterson variety 
\begin{equation}\label{eq:defP}
\Pet:= \overline{G^e. w_0B} \hookrightarrow G/B, 
\end{equation}
is the closure of the $G^e$-orbit of $w_0B$ inside the flag manifold $G/B$. 
\end{definition}

An alternate definition of Peterson varieties is provided as a special case of a larger class of varieties called {\em Hessenberg varieties,} parameterized by an element $x\in \mathfrak g$ and certain subspaces of $\mathfrak g$.  A Hessenberg space is a subspace $H\subseteq \mathfrak g$ satisfying $\mathfrak b\subseteq H$ and $[\mathfrak b, H]\subseteq H$. 
\begin{definition}\label{def:Hess}
The Hessenberg variety associated with $x\in \mathfrak g$, and a Hessenberg space $H$ is
\begin{equation}\label{eq:defHess}
Hess(x, H) = \{gB\in G/B: g^{-1}\cdot x\in H\},
\end{equation}
where $g^{-1}\cdot x$ indicates the adjoint action.
\end{definition}
We refer to regular, semisimple, or nilpotent Hessenberg varieties when $x$ is a regular, semisimple or nilpotent element, respectively, of $\mathfrak g$. 
Observe that when $H=\mathfrak g$,  $Hess(x, H)= G/B$, independent of $x$. Similarly when $x=0$, $Hess(x,H) = G/B$.

\begin{definition}
The Peterson variety is a regular nilpotent Hessenberg arising when $x=e$ as above, and
$$
H = H_0:= \mathfrak b\oplus \bigoplus_{\alpha\in \Delta} \mathfrak g_{-\alpha}.
$$ 
\end{definition}
A proof that $\Pet = Hess(e, H_0)$ is given in \cite{GolMihSin21}. 

\begin{example}
Let $G=Sl(2, \mathbb C)$, $B$ upper triangular matrices, $B^-$ lower triangular matrices, and $T = B\cap B^-$ the diagonal matrices. Then $\mathfrak g = H_0 =\mathfrak b \oplus \mathfrak g_{-\alpha}$, as there is only one simple positive root $\alpha$. Therefore, $\Pet = Sl(2,\mathbb C)/B.$
\end{example}

\begin{example} For $G = Sl(n,\mathbb C)$, there is an equivalent formulation for the Peterson variety. Let $g_i$ denote the $i$th column vector of the matrix $g\in G$, and let $V_i = \langle g_1, \dots, g_i\rangle$ denote the span of the first $i$ columns of $g$. Choose $e$ to be the element of $sl(n, \mathbb C)$ with $1$s above the diagonal, and $0$s elsewhere, i.e.
$$
e = \begin{pmatrix} 0 & 1 & 0 & \dots & 0&0\\ 0 & 0 & 1 & \cdots & 0 & 0 \\   \vdots &\vdots &\ddots& \ddots&\vdots &\vdots 
\\  
0 & 0 &  \cdots &0 & 1 & 0 \\ 
0 & 0 &\cdots &0 & 0 & 1 \\ 0 & 0&\cdots &0 &0&0
\end{pmatrix}.
$$ Then $gB\in Hess(x, H_0)$ if and only if $eV_i \subseteq V_{i+1}$ for $i=1, \dots, n-1$. For $n=3$, the elements $gB$ satisfying these conditions can be represented by matrices $g$ of any of the forms
\begin{align*}
\begin{pmatrix}
a & b & 1\\
b & 1 & 0\\
1 & 0 & 0
\end{pmatrix},
 \quad
\begin{pmatrix}
c& 1 & 0\\
1& 0 & 0\\
0 & 0 & 1
\end{pmatrix},
 \quad
\begin{pmatrix}
1 & 0 & 0\\
0 & d & 1\\
0 & 1 & 0
\end{pmatrix},\quad
\begin{pmatrix}
1 & 0 & 0\\
0 & 1 & 0\\
0 & 0 & 1
\end{pmatrix},
\end{align*}
for $a, b, c, d$ any complex numbers.
\end{example}

The cocharacter $h$ of $T$ satisfying $\alpha(h) =2$ for all $\alpha \in \Delta$
determines a one dimensional subtorus $S \subset T$
satisfying $\alpha|_S=\alpha'|_S$ for any $\alpha,\alpha'\in\Delta$.
Set $t:=\alpha|_S$ for $\alpha\in\Delta$;
then $t\in H^*_S(pt)$.

\section{Properties of $T$-equivariant cohomology for $G/B$}\label{se:ecohomologyG/B}

We use equivariant cohomology and equivariant Borel-Moore homology over $\mathbb Z$.  
We collect some useful facts about the $T$-equivariant cohomology of $G/B$, without proof or explanation. The reader may refer to \cite{Ful98} (Ch 19), \cite{Ful97} (Appendix B), \cite{CG97} (\S 2.6) for additional details about cohomology and Borel-Moore homology, and \cite{Gra01} for the equivariant versions.

\begin{property} The equivariant cohomology of a point is 
$$
H_T^*(pt) = S(\mathfrak h^*)\cong \mathbb Z[x_1,\dots, x_n]
$$
where $S(\mathfrak h^*)$ is the symmetric algebra in the dual of the Lie algebra of $T$, and $\dim T=n$. If one chooses a splitting of $T$, we may alternatively use polynomials of $\mathfrak h$. Here $x_i$ are degree 2 monomials.
\end{property}
\begin{property}\label{prop:inclusionfps}  The cohomology of the fixed point set $(G/B)^T$ is
$$H_T^*((G/B)^T) = \bigoplus_{v\in W}H_T^*(vB/B) = \bigoplus_{v\in W} \mathbb Z[x_1,\dots, x_n]
$$
\end{property}
\begin{property}\label{prop:localizationincl}  The inclusion map $(G/B)^T \hookrightarrow G/B$ induces an {\em inclusion}
$$
H_T^*(G/B) \hookrightarrow H_T^*((G/B)^T),
$$ 
called the {\em localization map.}
\end{property}
\begin{property}\label{prop:formality}  By using the ordinary cohomology on the Borel construction, the inclusion of $G/B$ as a fiber into the total space $G/B\times_T ET$ induces a 
 {\em surjection}
$$H_T^*(G/B) \twoheadrightarrow H^*(G/B),$$
called the {\em forgetful map.} This property is sometimes called {\em equivariant formality.}
%This map is often called the {\em forgetful map} as one can also obtain it by forgetting the $T$ action on $G/B$ and using o
\end{property}
\begin{property}  $H_T^*(G/B)$ is a free module over $H_T^*(pt)$. 
\end{property}

\begin{property}\label{prop:cohombasis}  A basis for the module $H_T^*(G/B)$ over $H_T^*(pt)$ is the collection of {\em equivariant Schubert classes} 
$$\{\sigma_v\in H_T^*(G/B):\  v\in W\}.$$
\end{property}

\begin{property}  The set  of Schubert varieties $\{[X_v]\}_{v\in W}$ is a basis for the ordinary homology $H_*(G/B)$ as a module over its coefficient ring $\mathbb Z$. An equivalent basis is given by the set of opposite Schubert varieties $\{[X^v]\}_{v\in W}$. In ordinary homology, $[X_v]= [X^{w_0v}]$.
\end{property}
\begin{property}\label{prop:schubTbasis}  The set  of Schubert varieties $\{[X_v]\}_{v\in W}$ is a basis for the equivariant homology $H^T_*(G/B)$ as a module over the equivariant cohomology of a point. Another basis is given by the set of opposite Schubert varieties $\{[X^v]\}_{v\in W}$. In equivariant homology, $[X_v]\neq [X^{w_0v}]$.
\end{property}

Each class $\sigma_v$ is dual to an {\em opposite} Schubert varieties in a specific sense: when $X^v$ is smooth, or more generally at the smooth locus, the corresponding Schubert class is the Euler class of the normal bundle to $X^v$. While $\sigma_v$ is often called the {\em Poincar\'e dual class} to $X^v$, one should not confuse this duality with the {\em Poincar\'e dual basis}. Indeed, the Poincar\'e dual basis of $\{\sigma_v\}_{v\in W}$ is given by $\{[X_v]\}_{v\in W}$, the equivariant homology classes corresponding to the Schubert varieties, not the opposite Schubert varieties.

This behavior between the bases $\{\sigma_v\}_{v\in W}$ and $\{[X_v]\}_{v\in W}$ of equivariant cohomology and equivariant homology, respectively, is formalized using the well-known pairing of classes on the manifold $G/B$. 
\begin{property}\label{prop:coSchubTbasis}
Let 
$$\langle \cdot , \cdot \rangle: H^*_T(G/B) \otimes H_*^T(G/B) \to H^*_T(pt)
$$ be defined by 
$\langle a,b \rangle = \int_b a.$
%\int_{G/B} a \cap b$.  
Then for all $v, w\in W$, 
$$
\langle \sigma_v , X_w\rangle = \delta_{v,w}.
$$
\end{property}

\begin{example} Let $G= Gl(3,\mathbb C)$, with $T$ the diagonal matrices. Consider the Schubert class $\sigma_v$ associated to $v=[231]$, obtained as the Poincar\'e dual class to the opposite Schubert variety $X^v$. By Proposition~\ref{prop:localizationincl}, $\sigma_v$ is determined by its restriction to the fixed point set. The restrictions can be found using \cite{Bil99}:
\begin{align*}
&\sigma_v\vert_{[123]}=0, \quad && \sigma_v\vert_{[213]}=0, \quad &&\sigma_v\vert_{[132]}=0\\
& \sigma_v\vert_{[231]}=\alpha_1\alpha_3 \quad &&\sigma_v\vert_{[312]}=0,\quad &&\sigma_v\vert_{[321]}= \alpha_1\alpha_3.
\end{align*}
Here $\alpha_1$ and $\alpha_2$ are the standard positive simple roots, and $\alpha_3=\alpha_1+\alpha_2$.
\end{example}

\section{Paving by Affines}\label{se:paving}

Recall the following definition: 
\begin{definition} 
%A variety $X$ is said to be {\em paved by affines} if it can be written as a finite disjoint union $X = \sqcup X_i^0$ where $X_i^0$ is isomorphic to affine space $\mathbb A^{d_i}$ for some $d_i$.
A variety $X$ is said to be {\em paved by affine spaces} or {\em paved by affines} if there is a sequence of closed subvarieties
$$
Y_0\subset Y_1\subset\cdots \subset Y_n=X
$$
such that $Y_i-Y_{i-1}$ is a finite, disjoint union of affine spaces.\footnote{Here we follow the definition provided by \cite{Pre13}, which does not restrict the dimension of the affine spaces occurring in $Y_i-Y_{i-1}$.}  We say that $X= \bigsqcup{Y_i-Y_{i-1}}$ {\em is a paving by affines.} 
% NOTE this definition is does not require that $X_i-X_{i-1}$ is a disjoint union of affine spaces of the same dimension.
\end{definition}

The flag variety $G/B$ has a paving by affines, as we show in the following example.
\begin{example}\label{ex:G/Bpaving} Consider the flag variety $G/B$ and Bruhat cells $X^0_w=BwB/B$ for $w\in W$. Bruhat cells are pairwise disjoint and cover all of $G/B$. Furthermore, each Bruhat cell $X^0_w$ is isomorphic to the affine space $\mathbb A^{\ell(w)}$. It is well known that $X_w = \overline{X^0_w}$ satisfies:
$$X_w= \bigsqcup_{u\in W, u\leq w} X^0_u, $$%\quad X_{w_0} = G/B,$$
%where $w_0$ is the longest element of $W$.  
Let $Y_i = \bigcup_{w\in W, \ell(w) = i}X_w$ be the union of those Schubert varieties of dimension $i$. Then $Y_i-Y_{i-1} = \bigsqcup_{w\in W, \ell(w)=i} X^0_w$ is a union of disjoint affine cells, and $G/B=\bigsqcup_{w\in W} X^0_w$ is a paving by affines. 
\end{example} 

De Concini, Lusztig and Procesi proved that Springer fibers are paved by affines \cite{deCLusPro88}. Tymoczko showed that all Hessenberg varieties in type A have an affine paving \cite{Tym07}. Precup generalized the result in \cite{Pre13} to all Lie types for regular Hessenberg varieties as well as a subset of operators whose nilpotent part are regular in a Levi subalgebra of $\mathfrak g$; we state her result here.

\begin{theorem}[Precup] Fix a Hessenberg space $H$ with respect to $\mathfrak{b}$.
\begin{enumerate}
\item Suppose $x \in \mathfrak{g}$ has Jordan decomposition $s+n$ with $s$ semisimple and $n$ regular in some Levi subalgebra  of the Lie algebra of the centralizer $Z_G(s)$. Then $Hess(x, H)$ has a paving by affines.
\item If $x$ is nilpotent and regular in a Levi subalgebra of  $ \mathfrak{g}$. Then, there is an affine paving of $Hess(x, H)$ given by the intersection of each Schubert cell in $G/B$ with $Hess(x, H)$.
\end{enumerate}
\end{theorem}

In particular, a variety with an affine paving has a cellular decomposition with only even-dimensional cells, and thus only even-dimensional homology. The closure of the affine cells form a basis for the homology; if the paving is $T$-invariant, then they form a basis for the $T$-equivariant homology. Dually we shall see (Theorem~\ref{thm:Arabia}) that one obtains a basis for the (equivariant or ordinary) cohomology ring.

Recall that $e$ is regular nilpotent, so Precup's result implies that $\Pet$ has an affine paving. We make this explicit in the case of Peterson varieties in Section \ref{se:positivityPeterson}. Peterson varieties are not smooth, so one cannot immediately extract a cohomological basis; Theorem \ref{thm:duality} states that there is nonetheless a dual basis.

\section{Positivity for $G/B$}\label{se:positivityG/B}
\subsection{Ordinary cohomology}
The ordinary cohomology ring $H^*(G/B)$ exhibits positive structure constants with a direct geometric interpretation.  Recall that each Schubert class $\sigma_v$ occurs as the {\em Poincar\'e dual class} to the variety $[X^v]$. As the Schubert classes form a basis for $H^*(G/B)$, we define structure constants $c_{u,v}^w\in H^*(pt)=\mathbb Z$  by the equation
\begin{equation}\label{eq:ordstructureconstants}
\sigma_u\sigma_v = \sum_{w\in W} c_{u,v}^w \sigma_w.
\end{equation}
We can identify the coefficients by pairing both sides with Schubert varieties. For any $q\in W$, 
\begin{align*}
\langle \sigma_u\sigma_v, [X_q]\rangle &=\langle  \sum_{w\in W} c_{u,v}^w \sigma_w, [X_q]\rangle \\
&=\sum_{w\in W} c_{u,v}^w\langle   \sigma_w, [X_q]\rangle =  c_{u,v}^q.
\end{align*}
On the other hand, we can obtain a geometric meaning to this product, by observing that the pairing counts intersection points.  
%Each cohomology class $\sigma_w$ is the Poincar\'e dual class to the homology class   $[X^w]$ of the opposite Schubert variety, in the sense that the classes map to each other under the natural pairing of cohomology and homology. 
%Furthermore, transverse intersection $gX^w\transv X_q$ (for $g\in G$ generic) is either a single point, or it is empty. 
The product $\sigma_u\sigma_v$ in cohomology is the algebraic representation of the intersection $X^u\cap gX^v$, for $g\in G$ generic. As long as $X^u\cap gX^v$ and $X_q$ intersect transversally and in a finite set of points, the pairing of the cohomology class $\sigma_u\sigma_v$ and the homology class $[X_q]$ is the number of points, counted with multiplicity and sign, in the triple intersection $X^u\cap gX^v \cap X_q$.  The intersection $X^u\cap X_q$ is known to be transverse \cite{Kle74}, and the intersection $X^u\cap gX^v\cap X_q$ is transverse  for generic $g$. The intersection is $0$-dimensional when $\ell(u)+\ell(v) = \ell(q)$. Finally we observe that $G$ acts on $G/B$ transitively, and preserving its complex structure. Thus all intersection points are positively oriented.  It follows that
$$
c_{u,v}^q = |X^u\cap gX^v\cap X_q|\geq 0
%|g_1X^u\transv g_2X^v\transv X_q| \geq 0.
$$
for $g$ in a dense open subset of $G$. 
 %%%
  In general, it's difficult to calculate intersection numbers of Schubert varieties using geometry and algebraic equations alone, though it can be done in small cases. The calculations are distinctly {\em easier} when using localization techniques. 
 \junk{
 \begin{example}
 In general, it's difficult to calculate intersection numbers of Schubert varieties using geometry alone, though it can be done in small cases. Let $G=Gl(3, \mathbb C)$, with $B$ and $B^-$ upper and lower triangular matrices in $G$, respectively. Choose $u=[213], v=[231], q=[321]$. Then $\ell(u)+\ell(v)=\ell(q)=3$, so the intersection $X^u\cap g X^v\cap X_q$
 is $0$-dimensional. We calculate its cardinality directly. Observe that $X^u$ and $X^v$ are the closures in $G/B$ of the left $B^-$ orbits given as cells:
 \begin{align*}
 C^u = \begin{pmatrix}
 0 & 1 & 0 \\
 1 & 0 & 0\\
c_1 & c_2 & 1
 \end{pmatrix}B \quad\mbox{and}\quad
  C^v = \begin{pmatrix}
 0 & 0 & 1 \\
1 & 0 & 0\\
a & 1& 0
 \end{pmatrix}B,
 \end{align*}
 respectively, 
 for $c_1, c_2, a$ any complex numbers. The variety $X_q = G/B$ since $q$ is the long word. Cleary $X^u\cap X_q=X^u$, and thus we need find the size of $X^u\cap gX^v$ for a generic choice of $g$. For $g = \begin{pmatrix} 
 g_{ij}\end{pmatrix}$, 
 $$
 gX^v =\overline{\left\{ \begin{pmatrix} g_{12}+ag_{13} &  g_{13} &  g_{11}\\ g_{22}+a g_{23}&  g_{23} &  g_{21}\\ g_{32}+a g_{33} &  g_{33} &  g_{31}\end{pmatrix}, \quad a\in \mathbb C \right\}}.
 $$
 In particular, it is a copy of $\mathbb C P^1$ given by the set of lines through the origin in the two-dimensional plane spanned by $\left\langle\begin{pmatrix} g_{12} \\ g_{22} \\ g_{32}\end{pmatrix}, \begin{pmatrix} g_{13}\\ g_{23}\\ g_{33}\end{pmatrix}\right\rangle$. 
 
 From here, one does a bit of linear algebra. A flag in the intersection satisfies both 
 $$\left\langle\begin{pmatrix} g_{12} \\ g_{22} \\ g_{32}\end{pmatrix}, \begin{pmatrix} g_{13}\\ g_{23}\\ g_{33}\end{pmatrix}\right\rangle = \left\langle\begin{pmatrix}  0\\ 1\\ c_1\end{pmatrix},\begin{pmatrix}  1\\ 0\\ c_2\end{pmatrix}\right\rangle \mbox{ and } \left\langle \begin{pmatrix} g_{12} \\ g_{22} \\ g_{32}\end{pmatrix} +a\begin{pmatrix} g_{13}\\ g_{23}\\ g_{33}\end{pmatrix}
 \right\rangle = \left\langle \begin{pmatrix} 0\\ 1\\ c_1\end{pmatrix}\right\rangle$$
 for some $a\in \mathbb C$. 
 Since $g$ is generic, we may presume that intersections occur in the large open cells $C^u$ and $gC^v$.  Some inspection of the resulting equations results in a single solution for $a$, which then determine $c_1$ and $c_2$. 
 Thus, there is only one flag in common between $gX^v$ and $X^u$, i.e. $c_{u,v}^q=1$.   These calculations are distinctly {\em easier} when using localization techniques. 
  \end{example}
  }

\subsection{Equivariant cohomology}
While positivity in the ordinary cohomology of Schubert calculus describes intersections, the equivariant version has additional subtleties, reflecting properties of the choice of Borel subgroup. 

Furthermore, equivariant positivity for Schubert calculus on $G/B$ (or $G/P$) generalizes ordinary cohomology, as we shall see.
%in the following sense.

The essential ingredients to positivity for the equivariant cohomology of $G/B$ may be parsed as coming from the following properties:
\begin{itemize}
\item 
{\bf Paving by Affines.} {\em $G/B$ has a paving by affine spaces.} While there are many such paving, we use  the  one given by Schubert cells $BwB/B$ for $w\in W$, as in Example~\ref{ex:G/Bpaving}.
\item 
{\bf $B$-Invariance of Paving.} {\em The chosen affine cells paving $G/B$ are $B$-invariant.} Their closures are the Schubert varieties $X_w$ for $w\in W$, and they constitute the effective basis guaranteed to exist by Theorem~\ref{thm:GrahamHom}. 
\item {\bf Positivity for Equivariant Homology.} {\em There is a set of $B$-invariant subvarieties in $G/B$ that provide an effective basis for its $T$-equivariant homology, with {\em positive} coefficients} for any $T$-invariant subvariety. This property holds for any scheme $X$ with a $B$ action, and is stated in Theorem~\ref{thm:GrahamHom}.
\end{itemize}
 %The varieties $X_w$ are $T$-invariant varieties, and indeed they are $B$-invariant varieties. By Property~\ref{prop:schubTbasis}, their homology classes $\{[X_w]_T\}_{w\in W}$ provide a basis for equivariant homology.  
 
 We first make explicit what we mean by positive.
Let $B$ be a choice of Borel in a complex reductive Lie group $G$, with unipotent radical $N$; let $T\subseteq B$ be a maximal torus. Then $T$ acts on $\mathfrak n= Lie(N)$ with weights $\alpha_1,\dots, \alpha_d$. These are the {\em positive roots} associated with the choice of $(B,T)$.  %We say that $\alpha_1,\dots, \alpha_d$ are a choice of positive roots. 
\begin{definition}[Graham positivity for equivariant cohomology] \label{def:positivity}
We say that a $T$-invariant scheme $X$ enjoys {\em Graham positivity} if there exists a choice of positive roots $\alpha_1,\dots, \alpha_d$ and a basis $\{\sigma_i\}$ of $H^*_T(X)$ as a $H_T^*(pt)$-module such that 
$$
\sigma_i\sigma_j = \sum_k c_{ij}^k \sigma_k,
$$
with $c_{ij}^k$ a linear combination of monomials in $\alpha_1,\dots, \alpha_d$ with nonnegative coefficients. We call the coefficients $c_{ij}^k$  {\em Graham positive.}
\end{definition}

Positivity for the equivariant homology of $G/B$ is stated in the following theorem due to Graham \cite{Gra01}. 
 
\begin{theorem}[Graham]\label{thm:GrahamHom}
Let $B'$ be a connected solvable group with unipotent radical $N'$, and let $T' \subset
B'$ be a maximal torus, so that $B' = T'N'$. Let $\alpha_1,\dots, \alpha_d$ be the weights of $T'$ acting on
$Lie(N')$. Let $X$ be a scheme with a $B'$-action, and $Y$ a $T'$-invariant subvariety of $X$. Then there
exist $B'$-invariant subvarieties $D_1, \dots, D_k$ of $X$ such that in the equivariant homology $H^{T'}_*(X)$,
$$
[Y]=\sum_i f_i [D_i],
$$
where each $f_i\in H_T^*(pt)$ is a linear combination of monomials in $\alpha_1,\dots, \alpha_d$ with nonnegative, integer coefficients.
\end{theorem}

We apply the theorem to the case that $X = G/B$ and $B'=B$,  and note that $\{X_w\}_{w\in W}$ are the set of $B$-invariant varieties, and thus must be the indicated varieties $D_i$. The weights of the $T$ action on $Lie(N')$ are the positive roots $\{\alpha_1, \dots, \alpha_d\}$ determined by our choice of Borel $B$, where $\dim N'=d,$ 
%$\dim T=d,$ 
and naturally live in $H_T^*(pt)$.

The theorem is fundamentally geometric in nature, in that it says that the homology class associated with an invariant subvariety of $G/B$ can be expressed as a Graham positive linear combination of homology classes associated to Schubert varieties.

\begin{corollary}\label{cor:positivecomboschubert}
Let $Y$ be a $T$-invariant subvariety of $G/B$. Then in $H^T_*(G/B)$, 
$$
[Y]_T=\sum_{w\in W} f_w [X_w]_T
$$
where $X_w$ are Schubert varieties and $f_w$ are nonnegative linear combinations of monomials in $\alpha_1,\dots, \alpha_d$ for all $w\in W$.
\end{corollary}

\begin{example} Let $G=Gl(3, \mathbb C)$, $B$ upper-triangular matrices in $G$, and $T$ diagonal matrices of the form 
$\begin{pmatrix} 
r_1e^{ix_1} & 0 &0\\
0 & r_2e^{ix_2} &0\\
0& 0 &r_3e^{ix_3}
\end{pmatrix}$, with nonzero real numbers $r_1, r_2, r_3.$ The unipotent matrices $U$ are upper triangular invertible matrices; $T$ acts on its  Lie algebra $\mathfrak u$ with weights 
$$\alpha_1=x_1-x_2, \ \alpha_2= x_2-x_3,\  \alpha_3=x_1-x_3.$$ The $B$-invariant basis of the homology of $G/B$ is given by the six Schubert varieties: $X_{[123]} = (id)B/B$ is a single point;  $X_{[213]}$ and $X_{[132]}$ are both two-dimensional spheres; $X_{[312]}$ and $X_{[231]}$ are both Hirzebruch surfaces ($\mathbb C P^1$ bundles over $\mathbb C P^1$) and $X_{[321]}=G/B$. 

Consider the $T$-invariant subvariety $Y$ in $G/B$ isomorphic to $\mathbb C P^1$ and given by the closure of the cell 
$$
\begin{pmatrix} 
0 & * &1\\
1& 0 &0\\
0 & 1 & 0
\end{pmatrix}B.
$$
By Corollary~\ref{cor:positivecomboschubert}, the class $[Y]_T$ may be expressed as a sum of positive polynomials in $\alpha_1,\alpha_2$ times Schubert varieties.  There are multiply ways to find this relationship. Consider the dual class $\gamma$ to $[Y]_T$, which is described by the restrictions $\gamma\vert_{[213]} = -\alpha_1\alpha_2$, $\gamma\vert_{[231]} = -\alpha_1\alpha_2$, and $\gamma\vert_w=0$ otherwise. Let $PD[X_w]_T$ indicate the equivariant Poincar\'e dual class to $X_w$. A check on the restrictions to fixed points affirms that 
$$
\gamma = \alpha_1 PD[X_{[231]}]_T + PD[X_{[132]}]_T,
$$ as the restrictions to fixed points are the same for both the left and right sides of the equation. Taking the dual, we obtain the equation in equivariant homology:
%In equivariant homology, we obtain
% $[Y]_T$ is equivalent to the sum
\begin{equation}\label{examplepositivesum}
[Y]_T=\alpha_1 [X_{[231]}]_T+[X_{[132]}]_T.
\end{equation}
Observe that the coefficients are nonnegative in $\alpha_1, \alpha_2, \alpha_3$. 

Alternatively, using the dual basis of Schubert classes, the restriction to fixed points results in the equation
$$
\gamma = -\alpha_2 \sigma_{[213]} + \sigma_{[312]}.
$$ 
The equivariant pushforward of  $\gamma\sigma_w$ is the coefficient of $[X_w]_T$ in the expansion; its calculation may be done using the Atiyah-Bott Berline-Vergne formula (see \cite{AtiBot84}). For example,
$$
\int\limits_{G/B} \gamma \sigma_{231} = \int\limits_{G/B} -\alpha_2 \sigma_{[213]}\sigma_{[231]}  + \sigma_{[312]}\sigma_{[231]} = -\alpha_2 + \alpha_3 = \alpha_1.
$$ 
accounting for the coefficient of $\alpha_1$ in front of $X_{[231]}$ in \eqref{examplepositivesum}. The other coefficients are found similarly. 
\end{example}

The duality between equivariant homology and equivariant cohomology stems from the paving by affines. The following theorem is proved by Arabia in \cite{Ara89}, and restated in \cite{Gra01}. It requires only that the paving be $T$-invariant, rather than $B$-invariant. 
\begin{theorem}[Arabia]\label{thm:Arabia}
Suppose the $T$-variety $X$ has a paving by $T$-invariant affines $X^0_i$. Then
\begin{enumerate}
\item $H^T_*(X)$ is a free $H_T^*$-module with basis $\{[X_i]_T\}$.
\item Suppose in addition that $X$ is complete and that $H_T^*(X)$ is torsion-free. Then there exist classes $x_i$ (of degree dim $X_i$) in $H_T^*(X)$ which form a basis for $H_T^*(X)$ as an $H_T^*$-module, such that the bases $\{[X_i]_T\}$ and $\{x_i\}$  are dual in the sense that $\int_{X_i} x_j=\delta_{ij}$.
\end{enumerate}
\end{theorem}
Note that Theorem~\ref{thm:Arabia} implies that there is a geometric basis for the cohomology when $X$ has a paving by affines, but it does not guarantee that the multiplication is Graham positive. In particular, it does not imply geometric positivity for varieties with affine pavings. 

Applied to $X=G/B$ and affines $X_w^0 = BwB/B$, Theorem~\ref{thm:Arabia}(1) implies Property~\ref{prop:schubTbasis}, whereas Theorem~\ref{thm:GrahamHom} only implies that $\{[X_w]_T\ |\ w\in W\}$ form a spanning set of $H^T_*(G/B)$. Furthermore, the duality of  Theorem~\ref{thm:Arabia}(2) implies the existence of the basis $\{\sigma_w\}_{w\in W}$ of $H_T^*(G/B)$ stated in Property~\ref{prop:cohombasis}, together with the pairing described in Property~\ref{prop:coSchubTbasis}.

We put these statements into two easy-to-reference corollaries:
\begin{corollary} The set $\{[X_w]_T\}_{w\in W}$ is a basis for $H^T_*(G/B)$ as a (free) module over $H_T^*(pt)$.
%as a module over $H_T^*(G/B)$.
\end{corollary}
\begin{corollary} The pairing $\int_{X_u}\sigma_v=\delta_{uv}$ defines a dual basis $\{\sigma_w\}_{w\in W}$ for $H_T^*(G/B)$ as a module over $H_T^*(pt)$. 
\end{corollary}

The cohomology classes 
 $\{\sigma_w\}_{w\in W}$ are called {\em Schubert classes}.
We now show that the positivity in homology, together with Property~\ref{prop:coSchubTbasis}, implies positivity in {\em cohomology}, following Graham's original argument in \cite{Gra01}. In particular, if $c_{uv}^w$ are defined by the equations $\sigma_u\sigma_v=\sum c_{uv}^w\sigma_w$, then as with the nonequivariant case,
\begin{align*}
\langle \sigma_u\sigma_v,[X_w]_T\rangle &=\langle\sum_{w'} c_{uv}^{w'}\sigma_{w'},[X_w]_T\rangle \\
&=\sum_{w'} c_{uv}^{w'}\langle\sigma_{w'},[X_w]_T\rangle= c_{uv}^w,
\end{align*}
by duality. On the other hand, if $\delta: X\hookrightarrow X\times X$ is the diagonal map, with projection maps $pr_1, pr_2: X\times X\longrightarrow X$ onto the first and second factors, respectively, then
\begin{align*}
\langle \sigma_u\sigma_v,[X_w]_T\rangle &=\langle \delta^*\left(pr_1^*\sigma_u\right) \delta^*\left(pr_2^*\sigma_v\right),[X_w]_T\rangle\\
&=\langle \delta^*\left(pr_1^*\sigma_u\cdot pr_2^*\sigma_v\right),[X_w]_T\rangle\\
&=\langle pr_1^*\sigma_u\cdot pr_2^*\sigma_v,\delta_*\left([X_w]_T\right)\rangle
 \end{align*}
 by the push-pull formula. 
 We apply Graham's theorem to the homology class $\delta_*\left([X_w]_T\right)\in H^T_*(G/B\times G/B)$. Let $X= G/B\times G/B$, with $B'=T\cdot N\times N$, and observe that the $B'$-invariant subvarieties of $G/B\times G/B$ are $X_u\times X_v$ for $u, v\in W$. The homology classes $\{[X_u\times X_v]\}_{u,v\in W}$ form a basis for the homology $H^T_*(G/B\times G/B)$ as a module over $H_T^*(pt)$. Therefore the equation
 \begin{equation}\label{hom:diagonal}
 \delta_*\left([X_w]_T\right) = \sum_{u',v'} a^{u',v'}_w [X_{u'}\times X_{v'}]
 \end{equation}
 uniquely determines the coefficients $a^{u',v'}_w\in H_T^*(pt)$. Thus 
 \begin{align*}
\langle \sigma_u\sigma_v,[X_w]_T\rangle &=\langle pr_1^*\sigma_u\cdot pr_2^*\sigma_v,\delta_*\left([X_w]_T\right)\rangle\\
&=\sum_{u',v'} a^{u',v'}_w\langle pr_1^*\sigma_u\cdot pr_2^*\sigma_v, [X_{u'}\times X_{v'}]\rangle\\
&=\sum_{u',v'} a^{u',v'}_w \langle \sigma_u, [X_{u'}]_T\rangle\cdot \langle \sigma_v, [X_{v'}]_T\rangle
 \end{align*} 
 where the last equality follows from the K\"unneth theorem and a careful analysis of the fibers of the maps (see \cite{Gra01} for details). Finally, we observe that each summand is $1$ if and only if both $u=u'$ and $v=v'$, and is $0$ otherwise, 
 %$\langle \sigma_u, [X_{u'}]_T\rangle=1$ if and only if $u'=u$ and is $0$ otherwise, 
 so that
 $$
 c_{uv}^w = \langle \sigma_u\sigma_v,[X_w]_T\rangle =a_w^{u,v}.
 $$
 Since $\delta_*\left([X_w]_T\right) =[\delta(X_w)]_T$, the left hand side of \eqref{hom:diagonal} represents a $T$-invariant subvariety of $X\times X$. By Theorem~\ref{thm:GrahamHom}, each coefficient $a^{u,v}_w$ 
is a linear combination of monomials in the weights of the action of $T$ on $Lie(N\times N)$. The list of distinct weights is the same as the weights of the action on $Lie(N)$, mainly the positive roots $\alpha_1,\dots, \alpha_d$, so $a^{u,v}_w$ are Graham positive. We state the theorem from \cite{Gra01} in full for completeness.
\begin{theorem}[Graham]\label{thm:GrahamCohom}
Let $B$ be a connected solvable group with unipotent radical $N$ and Levi decomposition $B = TN$. Let $\alpha_1, \dots, \alpha_d$ denote the weights of the $T$-action on $Lie(N)$. Let $X$  be a complete $B$-variety on which $N$ acts with finitely many orbits $X_1^0,\dots, X_n^0$. These are a paving of $X$ by $B$-stable affines; let $X_1,\dots, X_n$ denote the closures, so $\{[X_1]_T,\dots,[X_n]_T\}$ is a basis for $H^T_* (X)$. Let $\{x_1, \dots, x_n\}$ denote the dual basis of $H_T^*(X)$. Write
$x_i x_j = \sum_k a_{ij}^k x_k$ 
with $a_{ij}^k \in H_T^*(pt)$. Then each $a_{ij}^k $can be written as a sum of monomials $\alpha_1^{i_1}\cdots \alpha_d^{i_d}$
with nonnegative integer coefficients.
\end{theorem}
For ease of reference, we state as a corollary the positivity for the equivariant cohomology of $G/B$.
\begin{corollary}\label{cor:G/Bstructureconsts} Let $\{\sigma_w\}_{w\in W}$ be the basis of $H_T^*(G/B)$ consisting of Schubert classes. Then 
$$
 \sigma_u \sigma_v = \sum_{w\in W} c_{uv}^w \sigma_w
$$
defines structure constants $ c_{uv}^w $ that are Graham positive.
\end{corollary}
%In the case of $X=G/B$, the duality provided by Property~\ref{prop:coSchubTbasis} is obtained due to the paving by affines mentioned in Theorem~\ref{thm:GrahamCohom}. The analogous duality statement holds for any smooth complete variety with a paving by $B$-invariant affine varieties. It is valuable to note, however, that Theorem~\ref{thm:GrahamHom} does not require $X$ to be paved by $B$-invariant affine varieties.

\begin{example}
In type A, the product $\sigma_{[213]}^2= \alpha_1 \sigma_{[213]}+ \sigma_{[312]}$. To check the equality, one may evaluate both sides at each fixed point, and use Property~\ref{prop:localizationincl}. 
\end{example}

\begin{remark}[Natuality]
For any subtorus $S\subset T$, there is a natural projection of the dual to the Lie algebras, $\mathfrak{t}^*\rightarrow \mathfrak{s}^*$. This induces a map $\pi: H_T^*(G/B)\rightarrow H_S^*(G/B)$, that retains some positivity properties with respect to the restricted weights.  In particular, the set $\{ \pi(\sigma_w)\}_{w\in W}$ of $\pi(\sigma_w)\in H_S^*(G/B)$ form a basis of $H_S^*(G/B)$ over $H_S^*(pt)$. The product $\pi(\sigma_u)\pi(\sigma_v)$ results in structure constants that are themselves Graham positive in the weights $\pi(\alpha_1),\dots, \pi(\alpha_d)$.  These coefficients are, by abuse of notation, also called $c_{uv}^w$. 

When $S$ is the trivial torus, $\pi(\alpha_i)=0$ for all $i$, and the coefficients $c_{uv}^w$ are the ordinary structure constants defined by \eqref{eq:ordstructureconstants}: Graham positivity reduces to the statement that $c_{uv}^w$ are nonnegative numbers.
\end{remark}

\subsection{Localization}

Another notion of positivity is obtained by examining the localization map mentioned in Property~\ref{prop:localizationincl}. For 
any $w\in W$, let $\iota_w^*$ denote the composition
$$
\iota_w^*: H_T^*(G/B) \longrightarrow H_T^*((G/B)^T)= \bigoplus_{v\in W}H_T^*(vB) \longrightarrow H_T^*(wB).
$$
The image $\iota_w^*(\sigma_v)$ of a Schubert class $\sigma_v$ is called the {\em restriction of $\sigma_v$ to $w$}, or the {\em localization of $\sigma_v$ at $w$.} Observe that its value is in the equivariant cohomology of the point $wB$, and hence is a polynomial. 

\begin{theorem}[Billey \cite{Bil99}, AJS \cite{AndJanSoe94}] For all $v, w\in W$
The localization $\iota^*_w(\sigma_v)$ is Graham positive.
\end{theorem}
\begin{example}
Consider the  class $\sigma_{[312]}$, which is nonzero on fixed points $[312]B$ and $[321]B$ and zero elsewhere. The restriction to either fixed point is $\alpha_2\alpha_3$, which is clearly Graham positive. 
\end{example}

\section{Positivity for Peterson varieties}\label{se:positivityPeterson}

We show that Peterson varieties enjoy positivity properties similar to $G/B$, and in particular whether there's a geometric basis for its homology that translates into a positive product formula in cohomology. We noted in Section~\ref{se:defs} that $\Pet$ supports a $\mathbb C^*$ action by a one-dimensional subgroup $S\subset T$. The Peterson variety $\Pet$ does have a paving by $S$-invariant affine cells, which allows us to use Thereom~\ref{thm:Arabia} at liberty. We are limited in our application of Theorem~\ref{thm:GrahamCohom}, however as $\Pet$ is not $B$-invariant. Instead we use the $S$-invariant inclusion map $\Pet\hookrightarrow G/B$ to leverage positivity in $G/B$ for $\Pet$.

We begin by describing an affine paving for $\Pet$.
\begin{definition} Let $K\subset \Delta$ be a subset of simple roots. Let $W_K$ be the Weyl group associated to those simple roots, and $w_K\in W_K$ the longest word in $W_K$. The {\em Peterson Schubert cells} (or simply {\em Peterson cells})
are given by
$$
\Pet^\circ_K :=\Pet\cap Bw_K B/B.
$$
Define the {\em Peterson Schubert varieties} by $\Pet_K=\overline{\Pet_K^\circ}$.
\end{definition}

B{\u a}libanu proved that $\Pet^\circ_K\subset X_{w_K}$ is an affine space of dimension $|K|$ (see \cite{Bal17} and, for details, \cite{GolMihSin21}); $\Pet_K$ is consequently an irreducible subvariety of $X_{w_K}$. 

The Schubert cells restricted to $\Pet$ result in an $S$-stable affine paving on each $\Pet_K$:
$$
\Pet_K= \bigsqcup_{J\subseteq K} \Pet_J^\circ.
$$
By Theorem~\ref{thm:Arabia}(1)
%An $S$-invariant paving by affines implies that 
$H^S_*(\Pet)$ is a free module over $H_S^*(pt)$ with basis given by $\{[P_K]\}_{K\subseteq \Delta}.$ 
Furthermore, while $\Pet$ is not $B$ invariant, it is an $S$-invariant subvariety of $G/B$ where we may consider $G/B$ as a $B'$-stable variety for some Borel $B'$ containing $S$ as a maximal torus. Thus we may apply Theorem~\ref{thm:GrahamHom}, with $X=G/B$, $Y=\Pet$ and $T'=S$ to obtain the following corollary. Recall that $\pi(\alpha_i)=t$ for all $i=1, \dots, d$ since $S$ stabilizes the principle nilpotent. 
\begin{corollary}\label{co:graham} Let $K\subseteq \Delta$ be a subset of the simple roots. Consider the subvariety $P_K\hookrightarrow G/B$ and its expansion in terms of Schubert varieties:
\begin{equation}\label{eq:PetExp}
[P_K]_S = \sum_{w\in W} c_K^w [X_w]_S.
\end{equation}
The coefficients $c_K^w \in H^*_S(pt)$ are polynomials in $t$ with nonnegative coefficients.
\end{corollary}

As the notation may suggest, it turns out that $\Pet_K$ are themselves Peterson varieties for a smaller group, providing a stability property much like the one Schubert varieties $X_w$ exhibit for inclusions of flag varieties.\footnote{One has to take care in this statement, however, due to the action by $S$. Depending on the choice of an ambient flag variety $G'/B'$, the map $\Pet_K\hookrightarrow G'/B'\hookrightarrow G/B$ map not be $S$-invariant.}

Our first result is that there is a pairing between homology and cohomology of the Peterson variety that mimics the behavior of $G/B$. 
Consider the pairing 
$$\langle \cdot , \cdot \rangle: H^*_S(\Pet) \otimes H_*^S(\Pet) \to H^*_S(pt)
$$
 of equivariant cohomology and equivariant homology defined by 
${\langle a,b \rangle = \int_b a}$.
%\int_{\Pet} a \cap b}$.   
%{\em A priori} this pairing could  be degenerate on the singular variety $\Pet$, however Thereom~\ref{thm:duality} shows it is not. 
Let $\iota^*: H_S^*(G/B)\rightarrow H_S^*(\Pet)$ be the pullback in equivariant cohomology.  For each $K\subseteq \Delta$, a {\em Coxeter element} for $K$ is a permutation $v_K$ whose reduced word expression contains exactly one reflection for each element of $K$. 
The following theorem is proved in \cite{GolMihSin21}.
%{\color{blue} (cf.~\cref{thm:duals} below)}:

\begin{theorem}[Duality Theorem]
\label{thm:duality}
Let $I, J$ be subsets of the set of simple roots $\Delta$ and let $v_I\in W$ be any Coxeter element for $I$. Define $p_I = \iota^*\sigma_{v_I}$. 
Then
\[ \left\langle p_I, [\Pet_J]_S\right\rangle  = m(v_I) \delta_{I,J} \/, \]
where $m(v_I)$ is the multiplicity of the (unique) intersection point of $X^{v_I} \cap \Pet_I$. In particular, the pairing is nonnegative and nondegenerate.
\end{theorem}
%See~\cref{thm:duals} below.
The theorem is an algebraic consequence of the fact that the varieties $X^{v_I}$ and $\Pet_I$ intersect at a unique point,
namely $w_I$, the longest element in the subgroup $W_I$ determined by $I$.
Its proof exploits the poset structure of the affine paving by Peterson cells,
along with the duality of Schubert classes in $G/B$. It may be valuable to keep in mind that Theorem~\ref{thm:Arabia} guarantees a dual basis to $\{\Pet_K\}$; the main contribution of the Duality Theorem (Theorem~\ref{thm:duality}) is that each of the dual classes is a positive fraction of a pullback of a Schubert class. As a consequence, $\Pet$ enjoys geometric positivity.

The Duality Theorem has several consequences. It implies that
the equivariant push forward $\iota_*:H_*^S(\Pet) \to H_*^S(G/B)$ is injective. 
Injectivity for ordinary (non-equivariant) homology was proved in \cite{InsTym16}. It is also the key ingredient for the equivariant positivity of  {\em Peterson Schubert calculus,} meaning the coefficients obtained by the product of elements $p_K$ in $H_S^*(\Pet)$ are nonnegative (equivariant) structure constants, the same way that Schubert classes are for $H_T^*(G/B)$.

\begin{corollary}\label{Pbasis}
Choose a
Coxeter element $v_K$ for each subset $K\subseteq \Delta$, and define
$p_K:= \iota^*\sigma_{v_K} \in H^{2|K|}_S(\Pet).$
The classes 
$$\left\{\frac{p_K}{m(v_K)}\right\}_{K \subseteq \Delta}$$ form a 
$H^*_S(pt)$-basis of $H^*_S(\Pet)$. 
\end{corollary}
\begin{proof}
Divide by $m(v_K)$ and use the linearity of the pairing. 
\end{proof}

Corollary~\ref{Pbasis} implies that the expansion
\begin{equation}\label{eq:pullbackSchubertToPet}
\iota^*(\sigma_w) = \sum_{K\subseteq \Delta} b_w^K p_K
\end{equation}
uniquely defines structure constants $\{b_w^K\}$, since 
$\{p_K\}_{K\subseteq \Delta}$ forms a $H_S^*(pt)$-basis of $H_S^*(\Pet)\otimes \mathbb Q$.

\begin{corollary}\label{co:SchubExp} 
%Consider the expansion 
%\begin{equation}\label{eq:SchubExp}
%\iota^*(\sigma_w) = \sum_{K\subseteq \Delta} b_w^K p_K.
%\end{equation}
The coefficients $b_w^K$ in Equation~\eqref{eq:pullbackSchubertToPet} are monomials in $t$ with nonnegative coefficients.
\end{corollary}
\begin{proof} We pair Equation \eqref{eq:pullbackSchubertToPet} with $[P_L]_S$ to obtain
\begin{align*}
\langle \iota^*(\sigma_w), [P_L]_S\rangle & = \langle  \sum_{K\subseteq \Delta} b_w^K p_K, [P_L]_S\rangle\\
& =  \sum_{K\subseteq \Delta} b_w^K \langle   p_K, [P_L]_S\rangle\\
&=\sum_{K\subseteq \Delta} b_w^K m(v_K)\delta_{K,L} =b_w^L m(v_L),
\end{align*}
where the last line is due to the Duality Theorem (Theorem~\ref{thm:duality}).
On the other hand, using the push-pull formula 
\begin{align*}
\langle \iota^*(\sigma_w), [P_L]_S\rangle & =\langle \sigma_w, \iota_*([P_L]_S)\rangle \\
& = \langle \sigma_w,  \sum_{u\in W} c_L^u [X_u]_S\rangle \\
&= \sum_{u\in W} c_L^u \langle \sigma_w,   [X_u]_S\rangle = c_L^w,
\end{align*}
where the last equality is due to the duality in $H_S^*(G/B)$ of the bases $\{[X_u]\}$ and $\{\sigma_u\}$. Thus
$
c_L^w= b_w^L m(v_L).
$
In particular, since $c_L^w$ is a nonnegative polynomial by Corollary~\ref{co:graham}, and $m(v_L)$ is a nonnegative integer, $b_w^L$ is also a nonnegative polynomial. 
\end{proof}

%again for a specific choice of Coxeter elements $v_I$.} 

%The map $\iota_*:H_*(\Pet) \to H_*(G/B)$ was shown to be injective in the non-equivariant setting
%for a specific choice of Coxeter elements $v_I$ 
%by Insko and Tymoczko in \cite{insko.tymoczko:intersection.theory}.
%By the duality theorem, 
%the equivariant map $\iota_*:H_*^S(\Pet) \to H_*^S(G/B)$ is also injective. 

\begin{theorem}[Equivariant Positivity]
%\label{thm:mainintro}
Let $I, J, K$ be subsets of $\Delta$. 
The structure constants of multiplication, $c_{I,J}^K \in H^*_S(pt)$, given by
\begin{equation}
\label{eq:Pexp}
p_I p_J = \sum_K c_{I,J}^K p_K 
\end{equation} 
are polynomials in $t$ with nonnegative coefficients.
\end{theorem}
\begin{proof}
%We pair the product $p_I \cdot p_J$ with the variety $[P_L]$ for $L\subseteq \Delta$:
%\begin{align*}
%\langle p_I \cdot p_J, [P_L]_S\rangle &= \langle \sum_K c_{I,J}^K p_K, [P_L]_S\rangle \\
%&= \sum_K c_{I,J}^K\langle p_K, [P_L]_S\rangle \\
%& =  \sum_K c_{I,J}^K m(v_K)\delta_{K,L}\\
%& =  c_{I,J}^L m(v_L).
%\end{align*}

Observe that
\begin{align*}
p_I  p_J &= \iota^*\sigma_{v_I}\iota^*\sigma_{v_J}\\
&= \iota^*(\sigma_{v_I}\cdot \sigma_{v_J})\\
&= \iota^*\left(\sum_{w\in W} c_{v_I, v_J}^w\sigma_w\right)\\
&= \sum_{w\in W} c_{v_I, v_J}^w\iota^*(\sigma_w),
\end{align*}
where $c_{v_I, v_J}^w$ have the desired positivity properties, because these are structure constants occurring in the equivariant Schubert calculus of $G/B$ (Corollary \ref{cor:G/Bstructureconsts}). 
Together with Equation~\eqref{eq:pullbackSchubertToPet}, we obtain
\begin{align*}
p_I  p_J &=  \sum_{w\in W} c_{v_I, v_J}^w \sum_{K\subseteq \Delta} b_w^K p_K\\
&=  \sum_{w\in W} \sum_{K\subseteq \Delta} \left(c_{v_I, v_J}^w b_w^K\right) p_K.
\end{align*}
In particular,  the structure constants occurring in Equation~\ref{eq:Pexp} are
$$
c_{I, J}^K = \sum_{w\in W} c_{v_I, v_J}^w b_w^K.
$$
Nonnegativity of the coefficients appearing on the left hand side now follows from the fact that both factors in each summand are polynomials with nonnegative coefficients (Corollaries~\ref{cor:G/Bstructureconsts} and \ref{co:SchubExp}).
\end{proof}

While we have described why the coefficients are nonnegative, we have not provided a formula for their product.
Many formulas - including manifestly positive ones - exist for some of these structure constants. 
Harada and Tymoczko found an equivariant Monk rule in type A \cite{HarTym11}. A formulation of the results for ordinary cohomology was described using left-right diagrams in \cite{AbeHorKuwZen21}. 
Drellich generalized the equivariant Monk formula to all types, while the current author and Singh found Chevalley and Giambelli formulas in all types \cite{GolSin22}.

The current author and Gorbutt in \cite{GolGor22} generalized the Harada-Tymoczko rule to a positive combinatorial formula for all Peterson Schubert products in type A. To give a flavor, we provide the formula found in \cite{GolGor22} for $c_{I,J}^K$ for $I, J,$ and $K$  nonempty consecutive sequences in $\{1, \dots, n-1\}$, with $K\supseteq I\cup J$ and $|K|\leq |I|+|J|$. Denote the largest element of the consecutive set $J$ by $\mathcal H_J$ and the smallest element by $\mathcal T_J$. Then
$$
c_{I, J}^K = a! {\mathcal H_I - \mathcal T_J +1\choose a, \mathcal T_I-\mathcal T_K, \mathcal H_K-\mathcal H_J}
{\mathcal H_J - \mathcal T_I +1\choose a, \mathcal T_J-\mathcal T_K, \mathcal H_K-\mathcal H_I} t^a
$$
where $a= |I|+|J|-|K|$. Observe that the formula is manifestly nonnegative.

\section{Positivity for Regular Hessenberg varieties?}\label{se:Hess} 

The Peterson variety is but a single example of a regular Hessenberg variety $Hess(n,H)$, and all such Hessenberg varieties have been shown to support a paving by affine cells \cite{Pre13}.  Given the powerful results that may be derived from the paving by affines, one may wonder whether the resulting basis for the homology exhibits positivity. Indeed, Theorem~\ref{thm:Arabia} shows that the closure of the affines forms a basis for the homology, and by duality, the cohomology. One can alternatively ask whether the cohomology classes multiply in a positive way, i.e. whether the dual basis specified by Theorem~\ref{thm:Arabia} leads to positivity for the cohomology of $Hess(x, H)$. Suppose $Hess(x, H)$ has a paving by affines, with $\{\beta_a\}$ the basis for $H^*(Hess(x,H))$ guaranteed by Theorem~\ref{thm:Arabia}. The product structure in this basis is what we call {\em Hessenberg Schubert calculus.}
\begin{question} Suppose $Hess(x, H)$ has a paving by affines. 
Is Hessenberg Schubert calculus positive for $Hess(x,H)$?
\end{question}

There are circumstances in which one might expect to leverage the positivity for $G/B$. Let $n\in \mathfrak g$ be a regular nilpotent element. 
The natural map 
\begin{equation}\label{eq:cohomsurj} 
H^*(G/B) \rightarrow H^*(Hess(n, H))
\end{equation}
induced by the inclusion $Hess(n,H)\rightarrow G/B$ is a surjection \cite{AbeHorMasMurSat20}.  In contrast, when $x$ is not nilpotent, the map $H^*(G/B) \rightarrow H^*(Hess(x,H))$ is generally {\em not} a surjection, evidenced by very large Betti numbers in low degree.

It follows from \eqref{eq:cohomsurj} that a 
%$H^*(pt)\otimes \mathbb Q$ 
$\mathbb Q$  basis of $H^*(Hess(n,H))$ can be obtained by pulling back Schubert classes from $G/B$. 
Positivity of Schubert calculus for $G/B$ does not immediately imply the same for $Hess(n,H)$, however there is additional structure to exploit. Under the dual pairing, the map on homology groups is injective. Since a basis of $H_*(Hess(n,H))$ is provided by the affine paving, one can push forward these classes to $G/B$, do some footwork in $G/B$ and then pullback to $Hess(n,H)$ to obtain positivity. Indeed, this is the strategy we followed for the Peterson variety. 

The magic for Peterson varieties is that elements of the dual basis
to $\{[P_K]\}$ are positive combinations of pullbacks of Schubert classes (divided by a multiplicity). This leads one to the following refinement of the first question:

\begin{question}\label{quest:pullback} Let $n$ be a regular nilpotent element of $\mathfrak g$. 
Are elements {$\beta_a\in H^*(Hess(n,H))$ } of the dual basis to the paving by affines of $Hess(n, H)$ pullbacks of sums of positive coefficients times Schubert classes?
\end{question}
%Conjecturally, the answer is yes for regular nilpotent Hessenberg varieties $Hess(n,H)$.

The proof that $\eqref{eq:cohomsurj}$ is surjective requires identifying $H^*(Hess(n,H))$ as a $W$-invariant subalgebra of $H^*(Hess(s,H))$ for $s$ semisimple; the proof does not show that there is a geometric representative of the basis dual to the basis obtained by paving $Hess(n,H)$.

Nonequivariantly, Insko, Tymoczko and Woo showed that the cohomology class (and K-theory class) of a regular Hessenberg variety in type A is represented by a certain Schubert polynomial (Grothendieck polynomial). Combined with results due to \cite{EnoHorNagTsu22} showing that an additive basis in the nilpotent case can be achieved with smaller Hessenberg varieties, one may be able to see rather explicitly Graham positivity for the homology. The dual basis has not been described with Schubert polynomials. 

If $Hess(n,H)$ and its paving are invariant under a subtorus $S$ of $T$, one would expect an equivariant inclusion on homology in this case as well, resulting in a surjection $H^*_S(G/B)\rightarrow H_S^*(Hess(n,H))$ of equivariant cohomology.  As before, one can restrict the positive weights $\alpha_i$ to $S$ to obtain the right notion of ``positive" and then ask Question~\ref{quest:pullback} equivariantly. 

\begin{question} Is equivariant Hessenberg Schubert calculus for regular nilpotent Hessenberg varieties positive? 
\end{question}

In type A, the surjection in cohomology extends to Springer fibers ($H=\mathfrak b$), even when the nilpotent element $n$ is not regular; \cite{KumPro12}  shows this holds equivariantly as well. 
\begin{question} Is equivariant Hessenberg Schubert calculus for nonregular nilpotent Springer fibers positive? 
\end{question}

Another nilpotent case in which the cohomology is known to be surjective is the type A case in which $n$ is the matrix with $1$ in the top right corner entry of the matrix; in this case, the Hessenberg varieties are unions of Schubert varieties \cite{AbeCro16}.

In all these cases, the surjection in cohomology implies an injection 
$$
H^S_*(Hess(x,H))\rightarrow H^S_*(G/B)
$$
where Theorem~\ref{thm:GrahamHom} applies and to leads to positivity for homology.  
 if we could realize each of the dual classes $\beta_a$ as a Poincar\'e dual class to a specific subvariety of $G/B$ restricted to $Hess(x,H)$, we may be able to attain proof of positivity of the product in cohomology. Thus we close with the following question:
\begin{question}
For which nonregular nilpotent elements $n\in \mathfrak g$ is there an equivariant surjection
$$
H_S^*(G/B)\rightarrow H_S^*(Hess(n,H))?
$$
\end{question}

\bibliographystyle{amsalpha}
%\bibliography{biblio}

\end{document}